\documentclass[11pt,reqno,british,a4paper,twoside]{amsart}

\usepackage[main=british]{babel}
\usepackage{amssymb,mathtools,enumitem,tikz}
\usepackage{csquotes,microtype}

\usepackage[hidelinks]{hyperref}
\usepackage[capitalize]{cleveref}

\crefformat{section}{\S #2#1#3}
\Crefformat{section}{Section~#2#1#3}

\mathtoolsset{mathic=true}
\setlist[enumerate,1]{label=\normalfont(\roman*)}

\AddToHook{bfseries}{\boldmath}

\theoremstyle{plain}
\newtheorem{theorem}{Theorem}[section]
\newtheorem{proposition}[theorem]{Proposition}
\newtheorem{lemma}[theorem]{Lemma}

\theoremstyle{definition}
\newtheorem{remark}[theorem]{Remark}
\newtheorem{example}[theorem]{Example}

\numberwithin{equation}{section}

\NewDocumentCommand{\Lie}{m}{\mathrm{#1}}
\NewDocumentCommand{\lie}{m}{\mathfrak{#1}}

\NewDocumentCommand{\GL}{}{\Lie{GL}}
\NewDocumentCommand{\so}{}{\lie{so}}
\NewDocumentCommand{\SU}{}{\Lie{SU}}
\NewDocumentCommand{\Spin}{}{\Lie{Spin}}
\NewDocumentCommand{\su}{}{\lie{su}}
\NewDocumentCommand{\sP}{}{\lie{sp}}
\NewDocumentCommand{\Sp}{}{\Lie{Sp}}
\NewDocumentCommand{\spin}{}{\lie{spin}}
\NewDocumentCommand{\Un}{}{\Lie{U}}

\NewDocumentCommand{\Hodge}{}{\mathop{}\!{*}}

\NewDocumentCommand{\bC}{}{\mathbb{C}}
\NewDocumentCommand{\bR}{}{\mathbb{R}}
\NewDocumentCommand{\bH}{}{\mathbb{H}}
\NewDocumentCommand{\bZ}{}{\mathbb{Z}}

\NewDocumentCommand{\CP}{}{\bC P}

\DeclareMathOperator{\LIE}{Lie}
\DeclareMathOperator{\rank}{rank}
\DeclareMathOperator{\stab}{stab}
\DeclareMathOperator{\vol}{vol}
\DeclareMathOperator{\Span}{span}

\NewDocumentCommand{\Id}{}{\mathrm{Id}}

\DeclarePairedDelimiter{\abs}{\lvert}{\rvert}
\DeclarePairedDelimiter{\norm}{\lVert}{\rVert}
\DeclarePairedDelimiterX{\inp}[2]{\langle}{\rangle}{#1, #2}
\DeclarePairedDelimiter{\paren}{\lparen}{\rparen}
\DeclarePairedDelimiter{\halfopen}{\lbrack}{\rparen}

\NewDocumentCommand{\with}{}{\mid}
\NewDocumentCommand{\SetSymbol}{o}{\nonscript\:#1|
  \allowbreak\nonscript\:\mathopen{}}
\DeclarePairedDelimiterX{\Set}[1]{\lbrace}{\rbrace}{%
  \renewcommand{\with}{\SetSymbol[\delimsize]}%
#1 }

\NewDocumentCommand{\pd}{mm}{\partial #1 /\partial #2}

\NewDocumentCommand{\any}{}{\,\cdot\,}

\NewDocumentCommand{\eqbreak}{O{2}}{\\&\hspace{#1em}}
\NewDocumentCommand{\eqand}{O{1}}{\hspace{#1em}\text{and}\hspace{#1em}}

\newcommand{\hyphen}{\nobreak-\nobreak\hskip0pt}

\begin{document}

\title{Multi-toric geometries with larger compact symmetry}

\author{Thomas Bruun Madsen}

\address{School of Computing and Engineering\\
University of West London\\
St. Mary’s Road\\
Ealing\\
London W5 5RF\\
United Kingdom}

\email{thomas.madsen@uwl.ac.uk}

\author{Andrew Swann}

\address{Department of Mathematics and DIGIT\\
Aarhus University\\
Ny Munkegade 118, Bldg 1530\\
DK-8000 Aarhus C\\
Denmark}

\email{swann@math.au.dk}

\begin{abstract}
  We study complete, simply-connected manifolds with special holonomy that are
  toric with respect to their multi-moment maps.
  We consider the cases where there is a connected non-Abelian symmetry group
  containing the torus.
  For \( \Spin(7) \)-manifolds, we show that the only possibility are structures
  with a cohomogeneity-two action of \( T^{3} \times \SU(2) \).
  We then specialise the analysis to holonomy \( G_{2} \), to Calabi-Yau
  geometries in real dimension six and to hyperKähler four-manifolds.
  Finally, we consider weakly coherent triples on \( \bR \times \SU(2) \), and their
  extensions over singular orbits, to give local examples in the
  \( \Spin(7) \)-case that have singular orbits where the stabiliser is of rank
  one.
\end{abstract}


\maketitle

\section{Introduction}
\label{sec:introduction}

Torus symmetry has proved to be a powerful tool in the study of symplectic and
Kähler manifolds.
A first highlight is the Delzant construction \cite{Delzant:convex} of compact
toric symplectic manifolds from polytopes.
Here the polytope is obtained as the image of the symplectic moment map for the
\( T^{n} \)-action on~\( M^{2n} \).
Using the concept of multi-moment map \cite{Madsen-S:multi-moment} one can
extend some of these ideas to other geometries that are described by closed
differential forms.
In particular, for Ricci-flat geometries of holonomy \( G_{2} \), in dimension
seven, with \( T^{3} \)-symmetry, or of holonomy \( \Spin(7) \), in dimension
eight, with \( T^{4} \)-symmetry, concrete notions of multi-toric geometries
were introduced and studied in \cite{Madsen-S:toric-G2,Madsen-S:toric-Spin7}.
The image polytopes of the symplectic case are now replaced by tri-valent
graphs.
For the \( G_{2} \)-case infinitely many complete toric examples are now known
to exist: these include the Bryant-Salamon example on \( S^{3} \times \bR^{4} \)
\cite{Bryant-Salamon:exceptional}, further cohomogeneity-one examples as
classified by Foscolo, Haskins \& Nordström \cite{Foscolo-HN:G2TNEH} and
examples for toric Calabi-Yau manifolds of Foscolo~\cite{Foscolo:Spin7}.
These examples are all asymptotically conical or asymptotically locally conical.
In contrast, for the case of toric \( \Spin(7) \)-geometries, no complete
examples with full holonomy are currently known.

As the \( G_{2} \)-story suggests one might hope to find toric \( \Spin(7) \)
examples by looking for solutions with a symmetry group of greater dimension
containing the torus~\( T^{4} \).
In this paper, we study this possibility systematically.
The main conclusion is that for larger compact symmetry on simply-connected
manifolds, the geometry cannot be of cohomogeneity one, and the only potential
cases of enhanced symmetry are manifolds with a locally free action of
\( T^{3} \times \SU(2) \) acting with cohomogeneity two.

Our approach is mostly algebraic, considering the possible group actions.
First, we classify which compact groups containing a four-torus can act almost
effectively on \( \Spin(7) \) manifolds.
This gives a fairly rich set of cohomogeneity-one actions and the single
cohomogeneity-two case above.
We then impose requirements for the metrics to have holonomy \( \Spin(7) \) and
for the manifolds to be simply connected.
Discussion of the potential special orbits and implications for the graph
structure of the image of the multi-moment map, then rule out the
cohomogeneity-one cases.
On the other hand, we will show that there are local examples with interesting
structure in the cohomogeneity-two case.

We also discuss how these descriptions can be specialised to Ricci-flat
geometries of smaller dimension, including the \( G_{2} \), Calabi-Yau and
hyperKähler cases, of multi-toric type.
In the process we also describe some of the orbit structure of potential
cohomogeneity-two examples.

We begin the paper by describing how such geometries are related to each other,
and give a direct description of the toric graph structure.
\Cref{sec:spin7} recalls essential features of group actions on manifolds and of
the actions of subgroups of \( \Spin(7) \) on its irreducible eight-dimensional
representation.
In \cref{sec:multi-toric-actions} we show our main results about toric
\( \Spin(7) \)-manifolds, and in \cref{sec:spec-small-geom} we discuss the
specialisations and known results for the smaller Ricci-flat geometries.
Finally, in \cref{sec:local-constr-cohom}, we show how \( \SU(2) \)-invariant
weakly coherent triples on \( \bR \times \SU(2) \) can be used to construct examples
of \( \Spin(7) \)-structures with a cohomogeneity-two action of
\( T^{3} \times \SU(2) \) and certain singular orbits.

\paragraph{\bfseries Acknowledgements} We are grateful for partial support from
the Aarhus University Research Foundation.

\section{A Ricci-flat multi-Hamiltonian hierarchy}
\label{sec:ricci-flat-hierarchy}

Let \( T^{k} \) be a \( k \)-dimensional torus acting (almost) effectively on an
\( n \)-dimensional manifold~\( M \) preserving a collection
\( \alpha = (\alpha_{1},\dots,\alpha_{r}) \) of closed forms
\( \alpha_{i} \in \Omega^{p(i)+1}(M) \).
We call such an action \emph{multi-toric}.
We write \( \bR^{k} \) for the Lie algebra of \( T^{k} \) and
\( \Lambda^{m}\bR^{k} \) for the space of alternating \( m \)-linear forms
on~\( \bR^{k} \).
Specialising~\cite{Madsen-S:multi-moment}, a \emph{multi-moment map} for this
action is a collection of invariant maps
\begin{equation}
  \label{eq:multi}
  \nu = (\nu_{(1)},\dots,\nu_{(r)}),\quad
  \nu_{(i)}\colon M \to \Lambda^{p(i)}\bR^{k},
\end{equation}
such that
\begin{equation*}
  d (\nu_{(i)}(X_{1} \wedge X_{2} \wedge \dots \wedge X_{p(i)})) =
  \alpha_{i}(X_{1},X_{2},\dots,X_{p(i)}, \any),
\end{equation*}
for all vector fields \( X_{j} \) generated by the action of~\( T^{k} \); on the
left-hand side these are identified with the corresponding generating
elements~\( X_{j} \) in the Lie algebra \( \bR^{k} \) of~\( T^{k} \).
When \( \nu \) exists, we call the action \emph{multi-Hamiltonian}.
As \( T^{k} \) is compact and connected, and
\( \alpha_{i}(X_{1},\dots,X_{p(i)},\any) \) is a closed one-form, such a multi-moment
map exists whenever \( M \)~is simply connected, or more generally whenever
\( b_{1}(M) = 0 \).

We are interested in four Ricci-flat geometries of dimensions between four and
eight.
We begin by setting up their notation in a way that gives a uniform hierarchy,
with each geometry extending the previous one.
The rank~\( k \) of the relevant torus~\( T^{k} \) will be determined by the
condition that the dimension of the target space of the multi-moment
map~\( \nu \) agrees with the dimension of the quotient space \( M/T^{k} \), so is
equal to \( (\dim M) - k \).
To describe the multi-moment map, we usually choose bases for the vector
spaces~\( \Lambda^{p(i)}\bR^{k} \) and specify the components of the multi-moment map
with respect to these bases.

\smallbreak

\emph{HyperKähler four-manifolds}.
An almost hyperHermitian structure on a four-manifold is given by three
two-forms \( \omega_{1} \), \( \omega_{2} \), \( \omega_{3} \) which in a local compatible
orthonormal coframe \( e^{4}, e^{5}, e^{6}, e^{7} \) are given by
\begin{equation}
  \label{eq:hK-forms}
  \begin{gathered}
    \omega_{1} = e^{4} \wedge e^{5} + e^{6} \wedge e^{7},\quad
    \omega_{2} = e^{4} \wedge e^{6} + e^{7} \wedge e^{5},\\
    \omega_{3} = e^{4} \wedge e^{7} + e^{5} \wedge e^{6}.
  \end{gathered}
\end{equation}
By Atiyah \& Hitchin \cite{Atiyah-Hitchin:monopoles}, the structure is
\emph{hyperKähler} when all three two-forms \( \omega_{i} \) are closed.
The (restricted) holonomy of the corresponding Riemannian metric
\( g = \sum_{j=4}^{7} (e^{j})^{2} \) is then contained in \( \SU(2) = \Sp(1) \),
and each pair \( (g,\omega_{i}) \)~defines a Kähler structure.
The corresponding torus \( T^{k} \) has dimension one, so is just a circle
\( S^{1} \) generated by a single vector field~\( U \).
The components~\( \nu_{i} \) of the multi-moment map \( \nu \) are ordinary
symplectic moment maps, so satisfy \( d\nu_{i} = \omega_{i}(U,\any) \), for
\( i = 1,2,3 \).

The model example is \( M = \bC^{2} \) with coordinates \( x^{4},\dots,x^{7} \)
giving complex coordinates \( z = x^{4}+ie^{5} \), \( w = e^{6} + ie^{7} \), so
\( \omega_{1} = (i/2)(dz \wedge d\overline{z} + dw \wedge d\overline{w}) \),
\( \omega_{2} + i\omega_{3} = dz \wedge dw \).
For \( S^{1} \) acting by \( (z,w) \mapsto (e^{it}z, e^{-it}w) \), we have
\begin{equation*}
  \nu_{1} = \frac{1}{2}(\abs{w}^{2} - \abs{z}^{2}),\quad
  \nu_{2} + i\nu_{3} = i zw.
\end{equation*}

These geometries have been thoroughly studied via the Gibbons-Hawking Ansatz
introduced in~\cite{Gibbons-H:multi}.
In particular, all complete examples have been classified in
Bielawski~\cite{Bielawski:tri-Hamiltonian} when \( M \) has finite topology, and
in general in Swann~\cite{Swann:twist-mod}.
Interestingly, a consequence of these classifications is that the induced map
\( \nu\colon M^{4}/S^{1} \to \bR^{3} \) of the orbit space is a homeomorphism, when
\( M \) is complete.

\emph{Calabi-Yau geometry in real dimension six}.
An \( \SU(3) \)-structure on~\( M^{6} \) is specified by a two-form \( \omega \) and
a complex three-form \( \Omega_{\bC} = \Omega_{+} + i\Omega_{-} \) given in a local compatible
orthonormal coframe \( e^{2},\dots,e^{7} \) in terms of \eqref{eq:hK-forms} by
\begin{equation}
  \label{eq:CY-forms}
  \omega = e^{2} \wedge e^{3} - \omega_{1},\quad
  \Omega_{\bC} = \Omega_{+} + i\Omega_{-} = (e^{2} + ie^{3}) \wedge (\omega_{2} - i\omega_{3}).
\end{equation}
When \( \omega \) and \( \Omega_{\bC} \) are closed, the (restricted) holonomy
of~\( g \) is contained in~\( \SU(3) \); in particular, \( \omega \) and
\( \Omega_{\bC} \) are parallel for the Levi-Civita connection, \( g \)~becomes
Ricci-flat and we have a \emph{Calabi-Yau structure}.
For multi-toric geometry the torus \( T^{k} \) is of dimension two, \( k = 2 \).
The components of the multi-moment map \( \nu\colon M \to \bR^{4} \) may be chosen
to satisfy
\begin{equation}
  \label{eq:CYmm}
  d\nu_{i} = \omega(U_{i}, \any),\quad i=1,2, \eqand
  \nu_{3}+i\nu_{4} = \Omega_{\bC}(U_{1},U_{2},\any),
\end{equation}
when \( U_{1},U_{2} \) generate the \( T^{2} \)-action.

A model example is \( M^{6} = \bC^{3} \) with \( T^{2} \) acting as the maximal
torus in~\( \SU(3) \).
Using \( z^{1} = x^{2} + ix^{3} \), \( z^{2} = x^{4} - ix^{5} \),
\( z^{3} = z^{6} - iz^{7} \), we have
\( \omega = (i/2)\sum_{j=1}^{3} dz^{i} \wedge d\overline{z}^{i} \) and
\( \Omega_{\bC} = dz^{1} \wedge dz^{2} \wedge dz^{3} \).
The action of \( T^{2} \) is given by
\( (z^{1},z^{2},z^{3}) \mapsto (e^{it_{1}}z^{1}, e^{it_{2}}z^{2},
e^{-i(t_{1}+t_{2})}z^{3}) \), generating vector fields
\( U_{1} = \pd{}{t_{1}} \) and \( U_{2} = \pd{}{t_{2}} \).
For the components \eqref{eq:CYmm}, we have
\begin{equation}
  \label{eq:CYflatmm}
  \nu_{1} = \frac{1}{2}(\abs{z^{3}}^{2} - \abs{z^{1}}^{2}),\quad
  \nu_{2} = \frac{1}{2}(\abs{z^{3}}^{2} - \abs{z^{2}}^{2}),\quad
  \nu_{3} + i\nu_{4} = - z^{1}z^{2}z^{3}.
\end{equation}
These maps were considered by Gross~\cite{Gross:examples-sLag}, including the
observation that \( \nu \) provides a homeomorphism \( \bC^{3}/T^{2} \to \bR^{4} \).

If \( M^{4} \) is a hyperKähler manifold with multi-Hamiltonian
\( S^{1} \)-action, then the coframe choices above induce a Calabi-Yau structure
on \( M^{6} = S^{1} \times \bR \times M^{4} \) that is a Riemannian product with
multi-Hamiltonian \( S^{1} \times S^{1} = T^{2} \)-action.

\emph{Torsion-free \( G_{2} \)-structures in real dimension seven.} A
\( G_{2} \)-structure on \( M^{7} \) is specified by a certain
three-form~\( \varphi \), which in turn determines a metric, an orientation and the
Hodge dual four-form~\( \Hodge_{7} \varphi \), see Bryant~\cite{Bryant:exceptional}.
Using a local compatible orthonormal coframe \( e^{1},\dots,e^{7} \), we may
derive these forms from the expressions~\eqref{eq:CY-forms} as
\begin{equation}
  \label{eq:G2-forms}
  \varphi = e^{1} \wedge \omega - \Omega_{+},\quad \Hodge_{7} \varphi = e^{1} \wedge \Omega_{-} + \frac{1}{2}\omega^{2}.
\end{equation}
When \( \varphi \) and \( \Hodge_{7}\varphi \) are both closed, the associated metric has
(restricted) holonomy contained in~\( G_{2} \) and is Ricci-flat, see
Besse~\cite{Besse:Einstein} and Fernández \& Gray \cite{Fernandez-G:G2}.
The multi-toric examples carry an action of~\( T^{3} \), with multi-moment map
\( \nu\colon M \to \bR^{4} \) given by
\begin{equation}
  \label{eq:G2mm}
  \begin{aligned}
    d\nu_{i} &= \varphi(U_{j},U_{k},\any),\quad\text{for
             \( (i j k) = (1 2 3) \) as permutations,} \\
    d\nu_{4} &= (\Hodge_{7}\varphi)(U_{1},U_{2},U_{3},\any).
  \end{aligned}
\end{equation}

Some examples may be obtained as \( M^{7} = S^{1} \times M^{6} \), with
\( M^{6} \) multi-toric Calabi-Yau, and so the model example is
\( M^{7} = S^{1} \times \bC^{3} \) with the action of
\( T^{3} = S^{1} \times T^{2} \), where \( S^{1} \) rotates the first factor, and
\( T^{2} \) acts on~\( \bC^{3} \) as in the Calabi-Yau model.
Writing the first generating vector field as \( U_{1} \), and the other two as
\( U_{2} \) and \( U_{3} \), we get from \eqref{eq:G2mm} that
\begin{equation*}
  \nu_{1} + i\nu_{4} = \overline{z^{1}z^{2}z^{3}},\quad
  \nu_{2} = \frac{1}{2}(\abs{z^{2}}^{2}-\abs{z^{3}}^{2}),\quad
  \nu_{3} = \frac{1}{2}(\abs{z^{3}}^{2}-\abs{z^{1}}^{2}).
\end{equation*}

\emph{Torsion-free \( \Spin(7) \)-structures in real dimension eight.}
A \( \Spin(7) \)-structure on~\( M^{8} \) is specified by a four-form locally
given by
\begin{equation}
  \label{eq:Spin7-form}
  \begin{split}
    \Phi &= e^{0123} - (e^{01}+e^{23}) \wedge (e^{45}+e^{67})
        - (e^{02}+e^{31}) \wedge (e^{46}+e^{75}) \eqbreak
        - (e^{03}+e^{12}) \wedge (e^{47} + e^{56}) + e^{4567},
  \end{split}
\end{equation}
where \( e^{ijk\ell} = e^{i} \wedge e^{j} \wedge e^{k} \wedge e^{\ell} \).
For each \( x \in M \), the stabiliser of \( \Phi_{x} \) in the general linear group
\( \GL(8,\bR) \) is precisely the group \( \Spin(7) \).
The four-form~\( \Phi \)~determines a Riemannian metric~\( g \) for which the one
forms \( e^{0},\dots,e^{7} \) form an orthonormal coframe, see
Bryant~\cite{Bryant:exceptional}, or Karigiannis~\cite{Karigiannis:deformations}
for a more explicit expression.
When \( \Phi \) is closed, the holonomy of \( g \) reduces to a subgroup of
\( \Spin(7) \) and the metric is Ricci-flat, see Besse~\cite{Besse:Einstein} and
Fernández~\cite{Fernandez:Spin7-manf}.
Multi-toric examples carry an action of~\( T^{4} \) with multi-moment map
\( \nu\colon M \to \bR^{4} \) given by
\begin{equation}
  \label{eq:Spin7mm}
  d\nu_{i} = (-1)^{i+1}\Phi(U_{j},U_{k},U_{\ell},\any),
\end{equation}
where \( (ijk\ell) = (1234) \) as permutations.

In terms of the \( G_{2} \)-forms~\eqref{eq:G2-forms}, we have
\begin{equation*}
  \Phi = e^{0} \wedge \varphi + \Hodge_{7}\varphi.
\end{equation*}
If \( M^{7} \) is multi-Hamiltonian \( G_{2} \) with the forms
\eqref{eq:G2-forms}, then \( M^{8} = S^{1} \times M^{7} \) is multi-Hamiltonian
\( \Spin(7) \).
The model example is then \( T^{2} \times \bC^{3} \), with \( U_{1} \) and
\( U_{2} \) generators for the first two factors and \( U_{3}, U_{4} \)
generating the action of the subgroup \( T^{2} \) of \( \SU(3) \) on
\( \bC^{3} \).  From \eqref{eq:Spin7mm}, we have
\begin{equation*}
  \nu_{1} + i\nu_{2} = \overline{iz^{1}z^{2}z^{3}},\quad
  \nu_{3} = \frac{1}{2}(\abs{z^{3}}^{2} - \abs{z^{2}}^{2}),\quad
  \nu_{4} = \frac{1}{2}(\abs{z^{1}}^{2} - \abs{z^{3}}^{2}).
\end{equation*}

\smallbreak

We note that the multi-moment maps for the \( G_{2} \), Calabi-Yau and
hyperKähler cases can be obtained from the \( \Spin(7) \) case by suitable
reordering and sign changes.
For this reason we shall initially focus on the \( \Spin(7) \) case,
subsequently specialising to get results for the other geometries.

\subsection{The image graph}
\label{sec:image-graph}

In \cite{Madsen-S:toric-G2,Madsen-S:toric-Spin7} it was shown that the above
multi-Hamiltonian geometries have the property that when the torus action is
effective any non-trivial stabilisers are connected subtori of~\( T^{k} \) of
dimension at most two.
Note that an almost effective action of~\( T^{k} \) always induces an effective
action, since \( T^{k}/\Gamma \) for any finite subgroup~\( \Gamma \) is a compact
monogenic group, so again a torus.

For the image graph, let us consider the \( \Spin(7) \)-case; as already
suggested, this can subsequently be specialised to the other geometries in the
hierarchy.
The images in \( \bR^{4} \) of the set of points with a circle stabiliser are
straight lines.
Fixing an oriented primitive basis~\( E_{1},\dots,E_{4} \) of
\( \bR^{4} = \LIE(T^{4}) \), we have a dual basis \( e_{1},\dots,e_{4} \) and a
volume form \( \vol_{T} = e_{1} \wedge e_{2} \wedge e_{3} \wedge e_{4} \) for
\( \LIE(T^{4}) \).
Using this we may identify \( \Lambda^{3}\bR^{4} \) with \( \bR^{4} \) via the
inverse~\( \lambda_{T} \) of the isomorphism \( X \mapsto \vol_{T}(X,\any,\any,\any) \).

If \( X \) generates a circle subgroup of~\( T^{4} \), then each component of
the fixed-point set~\( M^{X} \) is mapped under \( \nu \) and~\( \lambda_{T} \) to a
straight line in \( \bR^{4} \) with tangent vector~\( X \).
At the image of a point with stabiliser \( T^{2} \), three such lines meet, so
they form a \emph{trivalent graph}, and the primitive tangent vectors in
\( \bZ^{4} \subset \bR^{4} \) to these three lines sum to zero.
This is a \emph{zero-tension} condition.

In the above references, a deeper result proved is that \( \nu \) induces a local
homeomorphism \( M/T^{4} \to \bR^{4} \).
However, if one is just interested in the structure of the graph, then the
following direct argument can be used.

For points with one-dimensional stabiliser generated by~\( X \), we choose our
generators of the \( T^{4} \)-action to have \( U_{1} = X \).
It then follows that \( d\nu_{i} = 0 \) along the fixed set~\( M^{X} \) for
\( i=2,3,4 \).
Thus the image of the fixed set is contained in a line
\( (\nu_{2},\nu_{3},\nu_{4}) = (c_{2},c_{3},c_{4}) \) for some
constants~\( c_{i} \).
Furthermore \( U_{2}, U_{3}, U_{4} \) are linearly independent on the fixed set,
giving that \( d\nu_{1} \ne 0 \), so \( \nu_{1} \) is monotone, and thus injective on
each component of~\( M^{X}/T^{4} \).

For a two-dimensional stabiliser, let \( p \in M \) be a point with stabiliser
generated by \( U_{1},U_{2} \).
We extend this to a basis of generators \( U_{1},U_{2},U_{3},U_{4} \) of
\( T^{4} \).
Let \( W \) be a neighbourhood of~\( p \) on which \( U_{3} \) and \( U_{4} \)
are linearly independent.
Now \( \Phi(U_{3},\any,\any,\any) \) induces a closed \( G_{2} \)-form on the
quotient~\( W/U_{3} \) of \( W \) by the circle generated by \( U_{3} \);
similarly, \( \omega = \Phi(U_{3},U_{4},\any,\any) \) is a symplectic form on the
quotient \( Z = W/(U_{3},U_{4}) \) of~\( Z \) by the two-torus generated by
\( U_{3} \) and~\( U_{4} \).
The action of \( U_{1} \) and \( U_{2} \) on the tangent space of the image of
\( p \) in~\( Z \) is that of \( T^{2} \) on \( \bC^{3} \) as a maximal torus in
\( \SU(3) \), cf.~\cref{sec:spin7}.
Hence, by the equivariant Darboux theorem, \( Z \) is equivariantly
symplectomorphic to an open set of the origin in~\( \bC^{3} \) with the
corresponding \( T^{2} \)-action.
Consequently, there are three sets with one-dimensional stabiliser meeting
at~\( p \), given by the complex coordinate axes in~\( \bC^{3} \).
Furthermore, the multi-moment maps \( \nu_{1} \) and \( \nu_{2} \) then correspond
(up to sign and order) to the symplectic moment maps \( \nu_{1} \) and
\( \nu_{2} \) of the flat Calabi-Yau model~\eqref{eq:CYflatmm}.
The primitive tangents to the images of the complex coordinate axes under
\( (\nu_{1},\nu_{2}) \) are then given by \( A_{1} = E_{1}+E_{2} \),
\( A_{2} = -E_{2} \) and \( A_{3} = -E_{1} \), up to overall scale.
These satisfy the zero-tension condition \( A_{1} + A_{2} + A_{3} = 0 \).

\section{Group actions on \texorpdfstring{\( \Spin(7) \)}{Spin(7)}-manifolds}
\label{sec:spin7}

For a compact group~\( G \) acting on a manifold~\( M \), each point
\( p \in M \) has a tubular neighbourhood equivariantly diffeomorphic to
\begin{equation*}
  G \times_{H} V = (G \times V)/H,
\end{equation*}
where \( H = \stab_{G}(p) \) is the stabiliser in \( G \) of~\( p \) and \( V \)
is the fibre of the normal bundle to the orbit \( G \cdot p \cong G/H \)
at~\( p \) as a linear representation of~\( H \).
The action of~\( H \) on \( G \times V \) is given by
\( (g,v)h = (gh,h^{-1}v) \).
In particular, the tangent representation of \( H \) at \( p \) is
\begin{equation*}
  T_{p}M \cong \lie{m} + V,
\end{equation*}
where \( \lie{g} = \lie{h} + \lie{m} \), with \( \lie{m} \) denoting the
isotropy representation.
Note that for a point \( q = [(g,v)] \) in this neighbourhood, the stabiliser is
\begin{equation*}
  \begin{split}
    \stab_{G}(q)
    &= \Set{g' \in G \with (g'g,v) = (gh^{-1},hv)\ \text{for some \( h \in H \)}} \\
    &= \Set{gh'g^{-1} \with h' \in H\;\text{and}\; h'v = v} \cong \stab_{H}(v).
  \end{split}
\end{equation*}
The set of fixed points \( V^{H} \) is a linear subspace of~\( V \) and is a
trivial \( H \)-module.
As \( H \) is compact, we have a splitting \( V = V^{H} + V' \) as
\( H \)-modules, where \( (V')^{H} = \Set{0} \).
We now see that \( \stab_{G}(q) \) is isomorphic to \( H \) via conjugation
in~\( G \) only if \( v \in V^{H} \), and is conjugate to a proper subgroup
of~\( H \) for \( v \in V\setminus V^{H} \).
It follows that there is an open dense set of points whose stabilisers are
conjugate; the corresponding points of the neighbourhood lie in \emph{principal
orbits} for the \( G \) action.
If \( p \) is a point in a principal orbit of~\( G \), then near \( p \) each
orbit is isomorphic to~\( G/H \), with \( H = \stab_{G}(p) \), so the normal
representation~\( V \) is trivial.

Now consider the case when \( M \) admits a \( \Spin(7) \)-structure that is
preserved by~\( G \).
Then for any point \( p \), we have that \( H = \stab_{G}(p) \) acts on the
tangent space via a restriction of the irreducible representation of
\( \Spin(7) \) on~\( \bR^{8} \).
In the following, it will be useful to determine which subgroups~\( H \) of
\( \Spin(7) \) admit a trivial summand~\( \bR^{k} \) in this representation.

From the list of transitive actions on spheres in Besse~\cite{Besse:Einstein},
we get the following information.
Firstly, \( \Spin(7) \) acts transitively on the sphere
\( S^{7} \subset \bR^{8} \) with stabiliser~\( G_{2} \).
Furthermore, \( G_{2} \) acts transitively on \( S^{6} \subset \bR^{7} \) with
stabiliser \( \SU(3) \), and \( \SU(3) \) acts transitively on
\( S^{5} \subset \bR^{6} = \bC^{3} \) with stabiliser~\( \SU(2) \).
The representation of \( \SU(2) \) on \( \bR^{5} \) is as \( \bR + \bC^{2} \),
and the action on \( S^{3} \subset \bC^{3} \) is transitive with trivial stabiliser.
We conclude that if the stabiliser~\( H \) has tangent representation:
\begin{enumerate}
\item \( T_{p}M = \bR + V \), then \( H \) is a subgroup of \( G_{2} \) such
  that \( V \) is the restriction of the standard representation of \( G_{2} \)
  on \( \bR^{7} \) to~\( H \);
\item \( T_{p}M = \bR^{2} + V \), then \( H \) is a subgroup of \( \SU(3) \),
  and \( V \) is a restriction of the standard representation of \( \SU(3) \) on
  \( \bC^{3} \);
\item \( T_{p}M = \bR^{3} + V \), then \( H \) is a subgroup of \( \SU(2) \),
  \( T_{p}M = \bR^{4} + V' \) and \( V' \) is a restriction of the standard
  representation of \( \SU(2) \) on \( \bC^{2} \);
\item \( T_{p}M = \bR^{5} + V \), then \( H \) and \( V \) are trivial,
  \( T_{p}M = \bR^{8} \).
\end{enumerate}

\section{Multi-toric actions on complete
\texorpdfstring{\( \Spin(7) \)}{Spin(7)}-manifolds}
\label{sec:multi-toric-actions}

Consider a non-compact multi-toric \( \Spin(7) \)-manifold \( M^{8} \) with a
compact symmetry group~\( G \) that is strictly larger than~\( T^{4} \).
We assume \( G \) acts almost effectively, is connected and is not Abelian.
We are most interested in the situation when \( M \) is complete with full
holonomy \( \Spin(7) \).
In the presence, of continuous symmetry this forces \( M \) to be non-compact,
justifying our previous assumption.

Principal orbits are of the form \( G/H \) and of codimension at least one, by
non-compactness of \( M \).
From the discussion in \cref{sec:spin7}, we have that \( H \) is isomorphic to a
subgroup of \( G_{2} \) and thus \( \dim H \leqslant 14 \) and
\( \dim(G/H) \leqslant 7 \).  This gives
\begin{equation*}
  \dim G \leqslant 7 + \dim H \leqslant 21.
\end{equation*}
For the action to be multi-toric we also have \( \rank G \geqslant 4 \).

We may write \( \lie{g} = \lie{a} + \bigoplus_{i=1}^{r}\lie{k}_{i} \) with
\( \lie{a} = \bR^{r} \) Abelian of rank~\( r \) and \( \lie{k}_{i} \) simple of
compact type.
As \( \dim G \leqslant 21 \) and \( \rank G \geqslant 4 \), we have
\( \dim\lie{k}_{i} \leqslant 21 \), and \( \dim\lie{k}_{i} < 21 \) if
\( \rank\lie{k}_{i} < 4 \), so \( \lie{k}_{i} \) is one of
\begin{equation}
  \label{eq:g-1}
  \su(4),\quad \lie{g}_{2},\quad \sP(2),\quad \su(3),\quad \su(2),
\end{equation}
with corresponding dimensions
\begin{equation*}
  15,\quad 14,\quad 10,\quad 8,\quad 3.
\end{equation*}
In particular, note that the algebras \( \so(7) \) and \( \sP(3) \) have
dimension \( 21 \) but only rank three, and therefore do not appear.

The possible Lie algebras \( \lie{h} \) of the subgroup \( H \) of \( G_{2} \)
are
\begin{equation*}
  \lie{g}_{2},\quad \su(3),\quad 2\su(2),\quad \bR + \su(2),\quad
  \su(2),\quad \bR^{2},\quad \bR,\quad 0.
\end{equation*}
with dimensions
\begin{equation*}
  14,\quad 8,\quad 6,\quad 4,\quad 3,\quad 2,\quad 1,\quad 0.
\end{equation*}

Given the list~\eqref{eq:g-1} of possible simple factors of~\( \lie{g} \), the
only factor that contains \( \lie{g}_{2} \) is \( \lie{g}_{2} \) itself.
As \( \dim \lie{g}_{2} = 14 > 21/2 \) there can be only one such factor
in~\( \lie{g} \), so if \( \lie{h} = \lie{g}_{2} \), then \( \lie{h} \) is an
ideal of \( \lie{g} \) and the action of this subalgebra is not effective.
Thus \( \lie{h} \) is a strict subgroup of \( \lie{g}_{2} \) and
\( \dim H \leqslant 8 \).

This gives
\begin{equation}
  \label{eq:dimg-2}
  \dim G \leqslant 7 + \dim H \leqslant 15.
\end{equation}
Note that whilst \( \dim\su(4) = 15 \), it has rank three, strictly less than
four, so we can remove it from the list~\eqref{eq:g-1} of candidates.

If \( \lie{g} \) contains \( \lie{g}_{2} \), then it follows
from~\eqref{eq:dimg-2} that \( \lie{g} \) is either \( \lie{g}_{2} \) or
\( \bR + \lie{g}_{2} \).
But these algebras have ranks two and three, so \( \lie{g}_{2} \) can also be
excluded from~\eqref{eq:g-1}.

Suppose \( \lie{g} \) contains \( \sP(2) \).
Then \eqref{eq:dimg-2} implies the other summands are either \( \su(2) \) or
Abelian.
None of these algebras contain \( \su(3) \), so in this case
\( \dim H \leqslant 6 \) and
\begin{equation}
  \label{eq:dimg-3}
  \dim G \leqslant 7 + \dim H \leqslant 13.
\end{equation}
Combining this with the constraint that \( \rank G \geqslant 4 \), we get
\( \lie{g} = \bR^{r} + \sP(2) \) with \( r = 2 \) or \( 3 \) and
\( \lie{h} = 2\su(2) \).
There is only one subalgebra \( 2\su(2) = \su(2)_{+} + \su(2)_{-} \) of
\( \sP(2) \), and we have
\begin{equation*}
  \sP(2)_{\bC} = S^{2}(\bC^{4}) = S^{2}(\bC^{2}_{+} + \bC^{2}_{-}) =
  S^{2}\bC^{2}_{+} + \bC^{2}_{+}\bC^{2}_{-} + S^{2}\bC^{2}_{-},
\end{equation*}
where \( \bC^{2}_{\pm} \) is the standard representation of \( \su(2)_{\pm} \)
on~\( \bC^{2} \).
Thus, as an \( \lie{h} \)-module, the tangent space \( T_{p}M \) for \( p \) in
a principal orbit is \( \bR^{4} + [S^{1}_{+}S^{1}_{-}] \).
As there is a trivial module of dimension four, the discussion of
\cref{sec:spin7} implies that \( \lie{h} = 2\su(2) \) is a subgroup
of~\( \su(2) \), which is clearly a contradiction.
Consequently, \( \lie{g} \) cannot contain \( \sP(2) \).

If \( \lie{g} \) contains \( \su(3) \), then as above it follows that it
contains only one such summand, since \( \dim\lie{g}\leqslant15 \).
For almost effective actions this implies that the isotropy algebra cannot
be~\( \su(3) \), so \( \dim H \leqslant 6 \), which again implies the
constraint~\eqref{eq:dimg-3}.
If \( \su(2) \) is also a summand then \( \lie{g} = \bR^{r} + \su(2) + \su(3) \)
with \( r = 1 \) or \( 2 \).
If \( \lie{h} = 2\su(2) \) then one summand of \( \lie{h} \) is an ideal of
\( \lie{g} \), and the action is not almost effective.
Thus \( \dim H \leqslant 4 \) and \( \dim G \leqslant 11 \); in particular, the above two
possibilities for \( \lie{g} \) do not occur.
It follows that \( \lie{g} = \bR^{r} + \su(3) \) with \( r = 2 \) or \( 3 \).
In either case, the isotropy group is at least of dimension three and so must
contain~\( \su(2) \).
If \( \lie{h} = \bR + \su(2) \), we get \( T_{p}M = \bR^{4} + \bC^{2} \), which
implies \( \lie{h} \) is a subgroup of \( \su(2) \); this is not possible.
Hence,
\begin{equation}
  \label{eq:pr-1}
  \lie{g} = \bR^{2} + \su(3),\quad \lie{h} = \su(2),
\end{equation}
with \( \su(3) = \su(2) + \bR + \bC^{2} \).
This is a highest root \( \su(2) \) subalgebra of~\( \su(3) \).
The standard three-dimensional representation of \( \su(3) \) splits as
\( \bC^{3} = \bC + \bC^{2} \) under the action of~\( \su(2) \).
This is an action of cohomogeneity one.

The remaining cases to consider are
\begin{equation*}
  \lie{g} = \bR^{r} + m\su(2)
\end{equation*}
where \( r + m \geqslant 4 \), \( r + 3m \leqslant 11 \) and \( m > 0 \).
Note that this gives \( m \leqslant 3 \).

For \( m = 3 \), we have \( r = 1 \) or \( 2 \) and hence
\( \dim H \geqslant 3 \).
Thus \( \lie{h} \) contains \( \su(2) \).
This is necessarily of subalgebra of \( 3\su(2) \leqslant \lie{g} \).
Furthermore \( \lie{h}\ne2\su(2) \), because \( \lie{h} \) contains no ideal
of~\( \lie{g} \).
It follows that the tangent representation \( T_{p}M \) is a sum of a trivial
representation~\( \bR^{2} \) with a copy of the tangent representation
of~\( 3\su(2)/\su(2) \).

If \( \su(2) \) is not diagonally embedded in \( 3\su(2) \), then
\( 3\su(2)/\su(2) \cong \bR^{3} + \su(2) \) and \( T_{p}M = \bR^{3} + V \) as an
\( \lie{h} \)-module, for \( \lie{h} = \su(2) \) or \( \bR+\su(2) \).
But the previous section now implies \( \lie{h} = \su(2) \) and
\( V = \bR + \bC^{2} \).
This does not contain the representation \( \su(2) \), so is not possible.

It follows that \( \su(2) \) is diagonally embedded in \( 3\su(2) \).
Hence, there is no Abelian subalgebra of \( 3\su(2) \) that commutes with
\( \su(2) \), so \( \lie{h} = \su(2) \) and we get the pair
\begin{equation}
  \label{eq:pr-2}
  \lie{g} = \bR + 3\su(2),\quad \lie{h} = \su(2)
\end{equation}
This is an action of cohomogeneity one.

If \( m = 2 \), the Lie algebra \( \lie{h} \) can not contain \( \su(2) \) as
otherwise the tangent representation would contain a single copy of the module
\( \su(2) \), which is not possible.
Thus \( \lie{h} \) is Abelian.
The tangent space now contains a trivial module of dimension at least four, so
\( \rank H = 1 \). Considering the rank of \( G \), we simply have
\begin{equation}
  \label{eq:pr-3}
  \lie{g} = \bR^{2} + 2\su(2),\quad \lie{h} = \bR,
\end{equation}
where the projection of \( \lie{h} \) to \( 2\su(2) \) has opposite weights in
the two summands.  This is an action of cohomogeneity one.

For \( m = 1 \), \( \lie{h} \) is a subalgebra of \( \bR \).
But it must be trivial, as for \( \lie{h}=\bR \) the only non-trivial submodule
of \( \lie{g}/\lie{h} \) is of dimension two.  Thus we get two cases,
\begin{gather}
  \label{eq:pr-45a}
  \lie{g} = \bR^{4} + \su(2),\quad \lie{h} = \Set{0}, \\
  \label{eq:pr-45b}
  \lie{g} = \bR^{3} + \su(2),\quad \lie{h} = \Set{0},
\end{gather}
with actions of cohomogeneity one and two, respectively.

\begin{proposition}
  \label{prop:Spin7-sym}
  Let \( G \) be a compact Lie group of rank at least four acting almost
  effectively and preserving a \( \Spin(7) \)-structure on a non-compact
  manifold~\( M \). Then either \( G \) is Abelian, or the principal orbit
  \( G/H \) is given by on of the pairs \( (\lie{g},\lie{h}) \) in equations
  \eqref{eq:pr-1}--\eqref{eq:pr-45b}.  \qed
\end{proposition}

We note that for the actions of cohomogeneity one, this list agrees with the
results of Reidegeld~\cite{Reidegeld:G2}, when restricted to non-Abelian groups
of rank at least four.

Now suppose that the restricted Riemannian holonomy is all of \( \Spin(7) \) and
the induced Riemannian metric on \( M \) is complete.
In the cohomogeneity-one case, we know from Mostert~\cite{Mostert:compact} that
the orbit space is homeomorphic to a circle or an interval.
As the metric is Ricci-flat, the Cheeger-Gromoll Theorem
\cite{Cheeger-G:nn-Ricci} implies that the orbit space is homeomorphic to
\( [0,\infty) \).
In particular, there is exactly one singular orbit~\( G/K \).
The quotient \( K/H \) is a sphere~\( S^{r} \), \( r \geqslant 1 \).
Note also that \( G/K \) is a deformation retract of~\( M \).

Let us suppose that \( M \) is simply connected.
Then the \( T^{4} \)-action is also multi-Hamiltonian and \( G/K \) is also
simply connected.
From the fibration \( S^{r} = K/H \to G/H \to G/K \), we have the exact sequences
\begin{equation}
  \label{eq:p1Sp1M}
  \pi_{1}(S^{r}) \to \pi_{1}(G/H) \to \pi_{1}(G/K) = \pi_{1}(M) = 1.
\end{equation}
Thus \( \pi_{1}(S^{r}) \) surjects to \( \pi_{1}(G/H) \), so \( \pi_{1}(G/H) \) is
either trivial, \( r > 1 \), or cyclic, \( r = 1 \).
By passing to a finite cover, we may assume that \( G = T^{k} \times G_{s} \) with
\( G_{s} \) a product of compact, simple, simply-connected groups.
As \( G/K \) is simply-connected and \( G \) is connected, we get that \( K \)
is connected.
For the cohomogeneity-one combinations above, we have for the centres \( Z(H) \)
and \( Z(G) \) that
\( \rank\lie{z}(\lie{h}) < \rank\lie{z}(\lie{g}) = \rank{T^{k}} = k \), with
\( k \geqslant 1 \), so \( \pi_{1}(G/H) \) is infinite, and \( r = 1 \).
This implies that \( \rank\lie{z}(\lie{k}) = 1 + \rank\lie{z}(\lie{h}) \).
For cases \eqref{eq:pr-1}, \eqref{eq:pr-45a} and \eqref{eq:pr-45b} this is not
possible with \( \pi_{1}(G/K) = 1 \).

For case \eqref{eq:pr-2}, we get that \( \lie{k} = \bR+\su(2) \).
As \( \lie{h} = \su(2) \) is the diagonal subalgebra of \( 3\su(2) \), there is
no one-dimensional subgroup of \( 3\su(2) \) that commutes with
\( \lie{h} = \su(2) \), so the projection of \( \bR \) in \( \lie{k} \) to the
summand \( 3\su(2) \) is trivial.
The tangent representation at \( p \) in the special orbit, is now
\( L + 2\su(2) \) with \( L \) an almost effective two-dimensional
representation of~\( \bR \), induced by a circle.
But then the action of~\( \bR \) fixes a six-dimensional subspace
of~\( T_{p}M \) and by \cref{sec:spin7}, this implies that the action is
trivial.  Hence, \eqref{eq:pr-2} does not occur.

This leaves the case~\eqref{eq:pr-3}, with \( \lie{k} = \bR^{2} \).
As \( G/K \) is simply-connected the projection of \( \lie{k} \) to the
\( \bR^{2} \)-summand is an isomorphism and we have \( K_{0} \cong T^{2} \) for the
component~\( K_{0} \) of the identity of~\( K \).
The map \( \pi_{1}(K) \to \pi_{1}(G) \) is a surjection
\( \bZ^{2} \to \bZ^{2} \), and so an isomorphism.
We get that \( T^{2} \hookrightarrow T^{2} \times \SU(2)^{2} \to K \) is an isomorphism.
In addition, \( \lie{k} \) contains~\( \lie{h} \), hence has at least a
one-dimensional projection to~\( 2\su(2) \).
The special orbit \( G/K \) is six-dimensional, so the normal representation
\( V \) is a non-trivial two-dimensional representation~\( L \) of~\( \lie{k} \)
with kernel~\( \lie{h} \).
On the other hand, the tangent representation of~\( \lie{k} \) contains a
trivial submodule of dimension at least one, implying that the tangent
representation is that of the maximal torus in \( G_{2} \) on
\( \bR + \bR^{7} \); this is conjugate to the action of the maximal torus in
\( \SU(3) \) on \( \bR^{2} + \bC^{3} \).
As a consequence, the projection of \( \lie{k} \) to \( 2\su(2) \) is
two-dimensional and the embedding of \( K \) is equivalent to
\begin{equation*}
  (e^{is},e^{it}) \mapsto (e^{is},e^{it},e^{i(as+bt)},e^{i(cs+dt)}),
\end{equation*}
with \( (a,b),(c,d) \in \bZ^{2} \) linearly independent over~\( \bR \).

We now have that \( G/K = S^{3} \times S^{3} \) via
\begin{equation*}
  (e^{is},e^{it},A,B) \mapsto (Ae^{-i(as+bt)},Be^{-i(cs+dt)}).
\end{equation*}
Regarding \( S^{3} \subset \bH = \bC + \bC j \) as the unit sphere, this space carries
a effective action of \( T^{4} \), induced by the action of \( G \), given by
\( (p,q) = (z+wj, u+vj) \mapsto (e^{it_{1}}z+e^{it_{2}}wj, e^{it_{3}}u+e^{it_{4}}vj)
\).
The action of any \( T^{4} \)-subgroup of \( G \) is conjugate to this action.
The point \( z+wj \) has non-trivial stabiliser if and only \( z = 0 \) or
\( w = 0 \) with the Lie algebra of the stabiliser generated by \( (0,1) \) or
\( (1,0) \), respectively, in \( \lie{t}^{2} = \bR^{2} \).
Thus the action of \( T^{4} \) on \( G/K \) has four sets where the stabiliser
is isomorphic to \( T^{2} \).  The Lie algebras of the four stabilisers are
\begin{equation*}
  \bR E_{1} + \bR E_{3},\quad \bR E_{2} + \bR E_{3},\quad
  \bR E_{2} + \bR E_{4},\quad \bR E_{1} + \bR E_{4},
\end{equation*}
with \( E_{i} \) denoting the standard basis vectors.
Note that together these stabilisers span \( \bR^{4} = \lie{t}^{4} \).
Taking this list cyclically the intersection of any two adjacent stabilisers is
the Lie algebra of the circle stabiliser of some points of
\( S^{3} \times S^{3} \).
In the image of the multi-moment map, each two-dimensional stabiliser defines a
point and these are joined by edges whose directions are parallel to the
one-dimensional stabilisers.
But the edges of a quadrilateral are necessarily linearly dependent, whereas the
one-dimensional stabilisers are generated by four linearly independent vectors
\( E_{1},\dots,E_{4} \).
This is a contradiction, so there is no multi-toric \( \Spin(7) \)-geometry
on~\( M \).
As \( \pi_{1}(M) = 1 \), any \( T^{4} \)-action preserving the structure is
multi-toric, so we conclude that there is no \( G \)-invariant structure
on~\( M \) of holonomy~\( \Spin(7) \).

Thus we have no examples of cohomogeneity one.
And are just left with the cohomogeneity-two case~\eqref{eq:pr-45b}.

\begin{theorem}
  Suppose \( M \) is a complete, simply-connected, eight-dimensional manifold
  with holonomy equal to \( \Spin(7) \), preserved by an effective action
  of~\( T^{4} \).
  If the geometry admits a non-Abelian, compact, connected, almost effective
  symmetry group \( G \) containing the \( T^{4} \)-action, then \( G \) is
  finitely covered by \( T^{3} \times \SU(2) \) acting with cohomogeneity two.  \qed
\end{theorem}

One way to obtain examples with reduced holonomy arises by taking products of a
circle with the \( G_{2} \)-solutions from Aslan \& Trinca
\cite[Theorem~4.8]{Aslan-T:G2-cohom2}.
This produces local solutions of cohomogeneity two under
\( T^{3} \times \SU(2) \) without singular orbits and with holonomy contained
in~\( G_{2} \leqslant \Spin(7) \).

\subsection{Cohomogeneity two}
\label{sec:cohomogeneity-two}

We now investigate the possible orbit structures of potential cohomogeneity-two
examples.
Here, we have \( G = T^{3} \times \SU(2) \) with \( H \) discrete.
We first determine the Lie algebras~\( \lie{k} \) for the possible singular
stabilisers~\( K \).

\begin{lemma}
  Suppose \( G = T^{3} \times \SU(2) \) acts almost effectively preserving a
  \( \Spin(7) \)-structure on \( M^{8} \).
  Then any non-trivial Lie algebra of a stabiliser~\( K \) is one of the
  following subalgebras of \( \lie{g} = \bR^{3} + \su(2) \):
  \begin{enumerate}
  \item \label{item:k1} \( \lie{k} \cong \bR \) with non-trivial projection to
    \( \su(2) \),
  \item \label{item:k2} \( \lie{k} \cong \bR^{2} \) with one-dimensional projection
    to \( \su(2) \),
  \item \label{item:k3} \( \lie{k} = \su(2) \), or
  \item \label{item:k4} \( \lie{k} = \bR + \su(2) \) with
    \( \bR \leqslant \bR^{3} \leqslant \lie{g} \).
  \end{enumerate}
\end{lemma}

\begin{proof}
  We have \( \lie{g} = \lie{k} + \lie{m} \), where \( \lie{m} \) is a
  complementary representation of~\( \lie{k} \), and the tangent model has the
  form \( T_{p}M = \lie{m} + V \) as a sum of \( \lie{k} \)-modules.
  This \( \lie{k} \)-module structure of \( T_{p}M \) needs to agree with that
  of restriction of the \( \Spin(7) \)-representation to a subalgebra isomorphic
  to~\( \lie{k} \).
  When this is satisfied we shall say that \( T_{p}M \) has type~\( \Spin(7) \).

  Suppose first that \( \dim\lie{k} = 1 \).
  If \( \lie{k} \) projects to \( \Set{0} \) in \( \su(2) \subset \lie{g} \), then
  \( \lie{m} = \bR^{5} \), and \( T_{p}M = \bR^{5} + V \), which is of
  \( \Spin(7) \) type only if it is a trivial \( \lie{k} \)-module.
  But then \( \lie{k} \) lies in the effective kernel of \( \lie{g} \),
  contradicting that \( G \) acts almost effectively.
  We thus have a non-zero projection of \( \lie{k} \) to \( \su(2) \).
  This gives \( \lie{m} = \bR^{3} + L \), with \( L \) a non-trivial
  two-dimensional representation and \( T_{p}M = \bR^{4} + \bC^{2} \) with
  \( \bC^{2} = L + L^{-1} \).

  For \( \dim\lie{k} = 2 \), suppose \( \lie{m} \) is trivial.
  Then \( T_{p}M = \bR^{4} + \bC^{2} \) with \( \lie{k} \) acting as a
  subalgebra of \( \su(2) \).
  But \( \su(2) \) only has rank one, so we can not have an effective action.
  Thus we have \( \lie{m} = \bR^{2} + L \), so \( \lie{k} \) has non-trivial
  projection to \( \su(2) \), and \( T_{p}M = \bR^{2} + L + V \) with
  \( L + V \cong \bC^{3} \) as a representation induced by the maximal torus
  in~\( \SU(3) \).

  When \( \dim\lie{k} = 3 \), we have either \( \lie{k} \) is Abelian or
  isomorphic to \( \su(2) \).
  The Abelian case does not occur, \( \lie{m} \) must contain a trivial
  submodule, so \( \lie{k} \) has to act on \( T_{p}M \) as subalgebra of
  \( \lie{g}_{2} \), which is only of rank two.
  Thus \( \lie{k} = \su(2) \), \( \lie{m} = \bR^{3} \) and
  \( T_{p}M = \bR^{4} + \bC^{2} \) with \( \bC^{2} \) the standard
  representation of \( \su(2) \).

  For \( \dim\lie{k} = 4 \), we have \( \lie{k} = \bR + \su(2) \) and
  \( \lie{m} = \bR^{2} \).
  It follows that \( T_{p}M = \bR^{2} + \bC^{3} \) with \( \lie{k} \) acting as
  a subalgebra of~\( \su(3) \).  This implies that \( K_{0} \) acts as
  \begin{equation}
    \label{eq:k4-u2}
    \Un(2) = \Set[\bigg]{
    \begin{pmatrix}
      A & 0 \\
      0 & \det A^{-1}
    \end{pmatrix}
    \with A \in \Un(2)
    }.
  \end{equation}

  If \( \dim\lie{k} = 5 \) then \( \lie{k} \) has rank three, but must act as a
  subalgebra of \( \lie{g}_{2} \), which is not possible.
  Similarly for \( \dim\lie{k} = 6 \), the rank is four, but the tangent action
  must be as a subalgebra of \( \spin(7) \), which is only rank three.
\end{proof}

\emph{Tubular neighbourhoods.}
We now turn to descriptions of the local models for the possible tubular
neighbourhoods
\begin{equation*}
  G \times_{K} V.
\end{equation*}

If \( \dim{K} = 4 \), then \( K_{0} = S^{1} \times \SU(2) \) with
\( S^{1} \leqslant T^{3} \leqslant T^{3} \times \SU(2) \).
As in Cox, Little \& Schenck \cite[Exercise~1.3.5]{Cox-LS:toric}, or using Smith
normal form, choosing a primitive generator \( v \) for
\( \lie{s}^{1} \leqslant \lie{t}^{3} \), we may extend~\( v \) to a
\( \bZ \)-basis of the integral lattice in~\( \lie{t}^{3} \).
As a consequence, we can assume that \( S^{1} \) is the last factor
of~\( T^{3} = S^{1} \times S^{1} \times S^{1} \).

The group \( K_{0} \) acts on \( V = \bC^{3} \) as in~\eqref{eq:k4-u2}.
In particular, the effective kernel contains \( (-1,-1_{2}) \), implying that
\( (1,1,-1,-1_{2}) \) must lie in~\( H \).
If \( K \ne K_{0} \), then \( K/K_{0} \leqslant G/K_{0} \cong T^{2} \) is generated by a
cyclic subgroup~\( Z \) of~\( T^{2} \) and has a representative in the first
factor of~\( T^{2} \times K_{0} \), so commuting with~\( K_{0} \).
The action of \( Z \) on \( V = \bC^{3} \) commutes with that of~\( K_{0} \) and
is contained in~\( \SU(3) \).
This implies that elements of~\( Z \) preserve the splitting
\( V = \bC^{2} + \bC \) of~\eqref{eq:k4-u2}, and the determinant~\( 1 \)
condition implies that \( Z \) acts as a subgroup of~\( \Un(2) \); so \( Z \)
must lie in its centre.
In particular, \( H \) contains a subgroup~\( Z' \) isomorphic to~\( Z \) that
commutes with~\( K \) and so is central in~\( G \).
The group \( H \) is Abelian, generated by \( Z' \) and the element
\( (1,1,-1,-1_{2}) \).
Note that the projection of \( H \) to \( \SU(2) \) is contained in
\( \Set{\pm1_{2}} \).
We may quotient \( G \) by the kernel of this projection to get a new group
isomorphic to \( G \), but with \( H \) isomorphic to~\( \bZ/2 \).
Thus we may assume \( K \) is connected and hence \( K = S^{1} \times \SU(2) \),
\( G = T^{2} \times K \) and \( H \cong \bZ/2 \), so \( G/H \cong T^{2} \times \Un(2) \).

For the action of \( \Un(2) \) on \( \bC^{3} \) given by~\eqref{eq:k4-u2}, the
stabilisers are
\begin{enumerate}
\item \( \Un(2) \) at~\( 0 \),
\item \( \SU(2) \) at \( (0,0,z) \), \( z\ne0 \),
\item isomorphic to \( \Un(1) = \Set{(1,e^{it},e^{-it})} \) at
  \( (w_{1},w_{2},0) \ne (0,0,0) \),
\item \( \bZ/2 \) otherwise.
\end{enumerate}
We thus get the following orbit structure near~\( p \)
\begin{enumerate}
\item the orbit through~\( p \) is \( G/K \cong T^{2} \),
\item there is a family of orbits parameterised by \( r = \abs{z} > 0 \) of the
  form \( G/\SU(2) \cong T^{3} \),
\item there is a family of orbits parameterised by
  \( r = \norm{(w_{1},w_{2})} > 0 \) of the form \( G/\Un(1) = T^{2} \times S^{3} \),
\item generic orbits close to~\( p \) are \( G/H \cong T^{2} \times \Un(2) \).
\end{enumerate}
The orbit space near \( p \) is a positive quadrant in~\( \bR^{2} \), with
\( \SU(2) \) stabilisers along one edge, and \( \Un(1) \) stabilisers along the
other edge.

For \( \dim K = 3 \), we have that \( K_{0} = \SU(2) \) and hence \( K \) splits
as \( Z \times K_{0} \) with \( Z \leqslant T^{3} \).
Thus \( G / Z \cong G \) and we may assume \( K = K_{0} \).
We have \( V = \bR + \bC^{2} \), with \( \bC^{2} \) as the standard
representation of~\( \SU(2) \).
Here all non-zero orbits are copies of \( S^{3} \cong \SU(2) \), so the stabilisers
on the tubular neighbourhood are \( \SU(2) \) along~\( \bR \times \Set{0} \) and
trivial otherwise.

Let us now consider the case when \( \dim K = 2 \).
Then \( K_{0} = T^{2} \) with non-zero projection to~\( \SU(2) \).
This acts on \( L + V = \bC^{3} \) as a maximal torus in~\( \SU(3) \).
Away from \( 0 \in V = \bC^{2} \) the non-trivial stabiliser occur along the
complex coordinate axes, and are each isomorphic to~\( \Un(1) \).
By the analysis for the case \( \dim K = 1 \), both of these generators have
non-trivial projection to~\( \su(2) \), so the corresponding \( G \)-orbits are
\( T^{2} \times S^{3} \) in families parameterised by \( r > 0 \).
The \( K_{0} \)-orbit space of~\( \bC^{2} \) is a positive quadrant
in~\( \bR^{2} \), with stabiliser \( T^{2} \) at the origin and
stabilisers~\( S^{1} \) along each edge.

\section{Specialisations to smaller geometries}
\label{sec:spec-small-geom}

Using the hierarchy of \cref{sec:ricci-flat-hierarchy}, we shall now discuss how
to apply the analysis of \cref{sec:multi-toric-actions} to explore the
lower-dimensional Ricci-flat geometries.

\subsection{\texorpdfstring{\( G_{2} \)}{G2}-manifolds}
\label{sec:G2-m}

Suppose \( M^{7} \) has a multi-toric \( G_{2} \)-structure with larger symmetry
group~\( G \).
By \cref{sec:ricci-flat-hierarchy}, this is equivalent to
\( M^{8} = S^{1} \times M^{7} \) having a multi-toric \( \Spin(7) \)-structure with
product symmetry \( S^{1} \times G \).
From~\cref{prop:Spin7-sym}, we immediately get the following.

\begin{proposition}
  \label{prop:G2-sym}
  Let \( G \) be a compact Lie group of rank at least three acting almost
  effectively and preserving a \( G_{2} \)-structure on a non-compact
  manifold~\( M^{7} \).
  Then either \( G \) is Abelian, or the principal orbit \( G/H \) is given by
  one of the following pairs \( (\lie{g},\lie{h}) \):
  \begin{enumerate}
  \item\label{item:G2SU3} \( (\bR + \su(3),\su(2)) \), with \( \su(2) \) a
    highest root subalgebra,
  \item\label{item:G23SU2} \( (3\su(2), \su(2)) \), with \( \su(2) \) embedded
    diagonally,
  \item\label{item:G22SU2} \( (\bR + 2\su(2),\bR) \), where the projection of
    \( \bR \) to \( 2\su(2) \) has opposite weights in the two summands,
  \item\label{item:G2SU2} \( (\bR^{3} + \su(2), \Set{0}) \),
  \item\label{item:G2coh2} \( (\bR^{2} + \su(2), \Set{0}) \).
  \end{enumerate}
  The last case is of cohomogeneity two, the others of cohomogeneity one. \qed
\end{proposition}

We now consider solutions with holonomy equal to \( G_{2} \).
In contrast to the \( \Spin(7) \)-case, cohomogeneity-one metrics with enhanced
symmetry are known to exist.

\begin{theorem}
  Suppose \( M \) is a complete, simply-connected, seven-dimensional manifold
  with holonomy equal to \( G_{2} \), preserved by an almost effective,
  non-Abelian, compact, connected, symmetry group \( G \) of rank at least
  three.
  Then \( G \) is finitely covered by \( \SU(2)^{3} \),
  \( S^{1} \times \SU(2)^{2} \) or \( T^{2} \times \SU(2) \).

  When \( G = \SU(2)^{3} \), the action is of cohomogeneity one and \( M \)~is
  isometric to the Bryant-Salamon structure \cite{Bryant-Salamon:exceptional} on
  \( S^{3} \times \bR^{4} \) realised as the spin bundle of~\( S^{3} \).

  When \( G = S^{1} \times \SU(2)^{2} \), the action is of cohomogeneity one and the
  geometries are classified in Foscolo, Haskins \& Nordström
  \cite[Theorem~E]{Foscolo-HN:G2TNEH}, giving solutions on
  \( S^{3} \times \bR^{4} \) and on \( M_{m,n} \), the total space of the circle
  bundle over \( K_{\CP(1)^{2}} \) with Chern class \( (m,-n) \).

  For \( G = T^{2} \times \SU(2) \), the action is of cohomogeneity two.
\end{theorem}

\begin{proof}
  The assumption that \( M \) is simply-connected and that \( G \) is rank at
  least three, ensures that there is a multi-Hamiltonian \( T^{3} \)-subaction.

  In the cohomogeneity-one cases \cref{prop:G2-sym}\ref{item:G2SU3}
  and~\ref{item:G2SU2}, the centre of \( H \) has smaller rank than the centre
  of~\( G \); the argument after~\eqref{eq:p1Sp1M} then implies that
  \( K/H = S^{1} \).
  For \ref{item:G2SU2} this is not compatible with \( \pi_{1}(G/K) = 1 \), so
  \cref{prop:G2-sym}\ref{item:G2SU2} does not occur.

  For \ref{item:G2SU3}, \( G = S^{1} \times \SU(3) \),
  \( K = S^{1} \times \SU(2) \) and \( H_{0} = \SU(2) \).
  The maximal torus~\( T \) of~\( G \) has rank three and may be chosen to
  contain the \( S^{1} \)-factor of~\( K \).
  We now consider the resulting graph under the multi-moment map and show that
  it can not be trivalent.

  Write \( \lie{t} = \bR E_{1} + \bR^{2} \) for the Lie algebra of~\( T \), with
  \( E_{1} \) spanning the Lie algebra of this~\( S^{1} \) and with
  \( \bR^{2} = \su(3) \cap \lie{t} \).
  We have \( G/H_{0} = S^{1} \times S^{5} \) with \( S^{5} \) as the unit sphere
  in the standard representation \( \bC^{3} \) of~\( \SU(3) \).
  The \( \bR^{2} \)-factor of \( \lie{t} \) acts only the \( S^{5} \)-factor and
  does so as the maximal torus \( T^{2} \) of~\( \SU(3) \).
  The special orbits of the \( T^{2} \)-action are thus the three circles that
  are the intersections of \( S^{5} \) with the complex coordinate axes
  in~\( \bC^{3} \).
  Let \( E_{2} \) and \( E_{3} \) generate two of these stabilisers, then for
  some choice of \( \varepsilon \in \Set{\pm1} \),
  \( E_{2} + \varepsilon E_{3} \) generates the third stabiliser.
  Since these span~\( \bR^{2} \subset \lie{t} \), it follows that in the finite
  quotient \( G/H \) of \( G/H_{0} \) these stabilisers are distinct.
  The one-dimensional stabilisers of the \( T \)-action are thus generated by
  \( E_{1} \), \( E_{2} \), \( E_{3} \) and \( E_{2} + \varepsilon E_{3} \).
  The first vector field vanishes only on \( G/K \), so there are three
  two-dimensional stabilisers spanned by \( \Set{E_{1},F} \) for
  \( F = E_{2} \), \( E_{3} \) and \( E_{2} + \varepsilon E_{3} \).
  As there are only two different types of one-dimensional stabiliser that meet
  at points with two-dimensional stabiliser, the resulting graph is only
  bivalent, and so case~\ref{item:G2SU3} can not extend to a simply-connected
  cohomogeneity-one structure with full holonomy.

  The remaining cohomogeneity-one cases are classified in Foscolo, Haskins \&
  Nordström \cite{Foscolo-HN:G2TNEH}.
\end{proof}

For the cohomogeneity-two case, solutions may be obtained by restricting the
symmetries of the cohomogeneity-one families.
General local solutions have been shown to exist in Aslan \& Trinca
\cite{Aslan-T:G2-cohom2}.

\subsection{Calabi-Yau manifolds in real dimension six}
\label{sec:CY6-m}

Suppose \( M^{6} \) has a multi-Hamiltonian \( \SU(3) \)-structure with larger
symmetry group~\( G \).
By \cref{sec:ricci-flat-hierarchy}, this is equivalent to
\( M^{7} = S^{1} \times M^{6} \) having a multi-Hamiltonian
\( G_{2} \)-structure with product symmetry \( S^{1} \times G \).
In this context, \Cref{prop:G2-sym} gives the following.

\begin{proposition}
  \label{prop:SU3}
  Suppose \( G \) is a compact Lie group of rank at least two acting almost
  effectively and preserving an \( \SU(3) \)-structure on a non-compact
  manifold~\( M^{6} \).
  Then either \( G \) is Abelian, or the principal orbit \( G/H \) is given by
  one of the following pairs \( (\lie{g},\lie{h}) \):
  \begin{enumerate}
  \item\label{item:CYSU3} \( (\su(3),\su(2)) \), with \( \su(2) \) a highest
    root subalgebra,
  \item\label{item:CY2SU2} \( (2\su(2),\bR) \), where \( \bR \) has opposite
    weights in the two summands,
  \item\label{item:CYSU2} \( (\bR^{2} + \su(2), \Set{0}) \),
  \item\label{item:CYcoh2} \( (\bR + \su(2), \Set{0}) \).
  \end{enumerate}
  The last case is of cohomogeneity two, the others of cohomogeneity one. \qed
\end{proposition}

\begin{theorem}
  Suppose \( M \) is a complete, simply-connected, six-dimensional manifold with
  holonomy equal to \( \SU(3) \) and an almost effective action of a
  non-Abelian, compact, connected, symmetry group \( G \) of rank at least two,
  preserving the \( \SU(3) \)-structure.
  Then \( G \) is finitely covered by \( \SU(3) \), \( \SU(2)^{2} \) or
  \( S^{1} \times \SU(2) \).
  The first two are actions of cohomogeneity one, and complete examples are
  known; the last is of cohomogeneity two.
\end{theorem}

\begin{proof}
  We have a multi-Hamiltonian action of~\( T^{2} \) on~\( M \).
  The arguments after~\eqref{eq:p1Sp1M} rule out case~\ref{item:CYSU2}
  of~\cref{prop:SU3}.
  We just need to recall existence of solutions in the other cohomogeneity one
  cases from the literature.

  For \( G = \SU(3) \), examples are provided by Calabi's metrics
  \cite{Calabi:kaehler} on the canonical bundle \( K \) of \( \CP(2) \).
  By Dancer \& Wang \cite{Dancer-W:KE}, these are the only irreducible examples.

  When \( G = \SU(2)^{2} \) examples are provided by the conifold metrics of
  Candelas \& de la Ossa \cite{Candelas-O:conifold} on the total space of
  \( 2\mathcal{O}(-1) \) over \( \CP(1) \) and by the metrics of
  Stenzel~\cite{Stenzel:Ricci-flat} on \( T^{*}S^{3} \).
  Restricting to the action of a subgroup \( S^{1} \times \SU(2) \) provides
  cohomogeneity two examples, that admit enhanced symmetry.
\end{proof}

\begin{remark}
  For \( G = S^{1} \times \SU(2) \) acting as Kähler isometries, the action is
  of cohomogeneity two, and non-trivial solutions on \( \bC^{3} \) have been
  constructed in Li~\cite{Li:CY-C3}.
  However, the \( S^{1} \)-factor does not preserve the complex volume form, but
  rotates it.
\end{remark}

\subsection{HyperKähler four-manifolds}
\label{sec:hyperk-four-manif}

Finally, we will briefly discuss complete hyperKähler four-manifolds with a
tri-Hamiltonian action of a compact non-Abelian group~\( G \).
If the principal orbit is~\( G/H \), then \( H \) acts trivially on the normal
bundle.
However, the tangent representation is a subrepresentation of \( \SU(2) \)
on~\( \bC^{2} \).
As this contains no trivial subrepresentation, \( \lie{h} = \Set{0} \) and
\( \dim G \leqslant 3 \).
We conclude that \( \lie{g} = \su(2) \).
By the general classification of hyperKähler four-manifolds with
\( \SU(2) \)-symmetry Atiyah \& Hitchin \cite{Atiyah-Hitchin:monopoles} and
Gibbons \& Pope \cite{Gibbons-P:positive} the metric is then the Eguchi-Hanson
metric on \( T^{*}\CP(1) \).

\section{Local construction for cohomogeneity two}
\label{sec:local-constr-cohom}

Returning to \( \Spin(7) \) manifolds, we have explained above that the only
candidates for simply-connected, multi-toric examples with larger symmetry have
\( G = T^{3} \times \SU(2) \).
To construct local examples, let us consider the framework of
Madsen~\cite{Madsen:Spin7T3} for the action of the \( T^{3} \)-factor.
If this factor is generated by \( U_{1},U_{2},U_{3} \) then we may consider the
multi-moment map \( \nu_{4} \) with
\( d\nu_{4} = - \Phi(U_{1},U_{2},U_{3},\any) \).
This has regular points wherever \( U_{1},U_{2},U_{3} \) are linearly
independent.
Madsen~\cite{Madsen:Spin7T3} shows that for each regular value~\( s \), the
quotient \( N = \nu_{4}^{-1}(s)/T^{3} \) is a smooth four-manifold that carries
a weakly coherent triple \( (\sigma_{1},\sigma_{2},\sigma_{3}) \) of symplectic
forms: by definition, each \( \sigma_{i} \) is symplectic and, fixing a volume
form~\( \vol \in \Omega^{4}(N) \) with the same orientation
as~\( \sigma_{1}^{2} \), the matrices
\( Q = (q_{ij}) \in \Omega^{0}(N,M_{3}(\bR)) \) defined by
\( \sigma_{i} \wedge \sigma_{j} = q_{ij}\vol \) are positive definite at each
point.
Given a two-form \( F \in \Omega^{2}(N,\bR^{3}) \) with integral periods, one
may define \( R = (r_{ij}) \in \Omega^{0}(N,M_{3}(\bR)) \) by
\( F_{i} \wedge \sigma_{j} = r_{ij}\vol \).
Theorem~3.8 of \cite{Madsen:Spin7T3} then asserts that \( \Spin(7) \)-geometries
with \( T^{3} \)-symmetry may be constructed from the \( T^{3} \)-bundle
over~\( N \) with curvature form~\( F \) and a geometric flow, provided \( R \)
is symmetric, \( R = R^{T} \), at each point.

This construction is equivariant for the action of the \( \SU(2) \)-factor
of~\( G \). Indeed, the multi-moment map \( \nu_{4} \) is
\( \SU(2) \)-invariant.
Moreover, the symplectic forms \( \sigma_{i} \) pull-back to the restriction of
\( \Phi(U_{j},U_{k},\any,\any) \), where \( (ijk)=(123) \) as permutations, on
the level sets of~\( \nu_{4} \), and so are invariant.
Additionally, the connection forms \( \vartheta_{i} \) for the
\( T^{3} \)-bundle are metric dual to \( U_{i} \) and zero on the orthogonal
complement of \( \Span\Set{U_{1},U_{2},U_{3}} \), so these are invariant.
Hence, the curvature \( F = (d\vartheta_{1},d\vartheta_{2},d\vartheta_{3}) \)
descends to an \( \SU(2) \)-invariant form.

We conclude that \( \Spin(7) \)-manifolds~\( M \) with
\( G = T^{3} \times \SU(2) \)-symmetry, can be constructed from
four-manifolds~\( N \) with an \( \SU(2) \)-invariant weakly coherent triple
\( (\sigma_{1},\sigma_{2},\sigma_{3}) \) and an \( \SU(2) \)-invariant curvature
form~\( F \).
As \( G \)~acts locally freely on open set of~\( M \), we may choose our regular
value~\( s \) so the \( \SU(2) \)-action on an open set of
\( N = \nu_{4}^{-1}(s)/T^{3} \) is also locally free.
So an open set of \( N \) carries a cohomogeneity one-action of~\( \SU(2) \).

Consider \( N = \bR \times \SU(2) \).
Write \( t \) for the parameter on the \( \bR \)-factor and let
\( \theta_{1}, \theta_{2}, \theta_{3} \) be a left-invariant basis of one-forms
on \( \SU(2) \) satisfying
\begin{equation*}
  d\theta_{1} = \theta_{2} \wedge \theta_{3},
\end{equation*}
etc.
We then have a left-invariant coframe on \( N \) given by
\( dt, \theta_{1}, \theta_{2}, \theta_{3} \).
Consequently, a general left-invariant two-form on~\( N \) can be expressed as
\( \sigma_{i} = \sum_{(pqr)=(123)} a_{pi} dt \wedge \theta_{p} + b_{pi}
\theta_{q} \wedge \theta_{r} \) for \( a_{pi} \) and \( b_{pi} \) functions
of~\( t \).
Such a form has differential
\( d\sigma = \sum_{(pqr)=(123)} (- a_{pi} + \dot{b}_{pi}) dt \wedge \theta_{q}
\wedge \theta_{r} \), and so is closed only if \( a_{pi} = \dot{b}_{pi} \), for
\( p = 1,2,3 \).  The resulting form
\begin{equation}
  \label{eq:sigma}
  \sigma_{i} = \sum_{(pqr)=(123)} \dot{b}_{pi} dt \wedge \theta_{p} + b_{pi} \theta_{q} \wedge \theta_{r}.
\end{equation}
is symplectic if and only if
\begin{equation*}
  \partial_{t}(b_{1i}^{2}+b_{2i}^{2}+b_{3i}^{2}) \ne 0
\end{equation*}
everywhere.
We have
\( \sigma_{i} \wedge \sigma_{j} = \sum_{p=1}^{3} \partial_{t}(b_{pi}b_{pj}) \vol
\), where \( \vol = dt \wedge \theta_{1} \wedge \theta_{2} \wedge \theta_{3} \).
It follows that \( (\sigma_{1},\sigma_{2},\sigma_{3}) \) is a weakly coherent
triple if \( B = (b_{pi})_{p,i=1}^{3} \) satisfies that
\begin{equation}
  \label{eq:pd}
  \partial_{t}(B^{T}B) > 0,
\end{equation}
meaning that this derivative is positive definite at each~\( t \).
The other possibility to get a weakly coherent triple is given by
\( \partial_{t}(B^{T}B) < 0 \), but this may be obtained from the positive
definite case by changing the sign of~\( t \).

For a left-invariant curvature form \( F = (F_{1},F_{2},F_{3}) \), we may write
each component as
\begin{equation*}
  F_{i} = \sum_{(pqr)=(123)} \dot{c}_{pi} dt \wedge \theta_{p} + c_{pi} \theta_{q} \wedge \theta_{r},
\end{equation*}
with \( c_{pi} \) functions of~\( t \).
Writing \( C = (c_{pi})_{p,i=1}^{3} \), we then get the data to define a
\( T^{3} \times \SU(2) \)-invariant \( \Spin(7) \)-structure if
\begin{equation}
  \label{eq:cb}
  \partial_{t}(C^{T}B)\ \text{is symmetric}
\end{equation}
together with \eqref{eq:pd}, for each~\( t \).

The two conditions \eqref{eq:pd} and \eqref{eq:cb} are not very restrictive and
give plenty of flexibility.
Note that \eqref{eq:pd} implies that \( B \) is invertible: Given a non-zero
\( x \in \bR^{3} \) constant, put \( b = Bx \).
Then
\( 0 < x^{T} \paren[\big]{\partial_{t}(B^{T}B)} x = \partial_{t}\norm{b}^{2} =
2\inp{\dot{b}}{b} \), so \( b \ne 0 \) and the null space of \( B \)
is~\( \Set{0} \).  We may thus write the general solution to \eqref{eq:cb} as
\begin{equation*}
  C = (B^{T})^{-1}(S(t) + D)
\end{equation*}
with \( S(t) \) symmetric and \( D \) constant.

\begin{example}
  One simple non-diagonal example, so not of hyperKähler type, is given by
  \begin{equation*}
    B =
    \begin{pmatrix}
      e^{t} & 1 & 0 \\
      0 & e^{t} & 1 \\
      0 & 0 & e^{t}
    \end{pmatrix}
  \end{equation*}
  for \( t > - (\log 2)/2 \), since
  \( \partial_{t}(B^{T}B) = 2e^{2t}\Id_{3} + e^{t}X \) with \( X =
  \begin{psmallmatrix}
    0 & 1 & 0 \\
    1 & 0 & 1 \\
    0 & 1 & 0
  \end{psmallmatrix}
  \) and the smallest eigenvalue of \( X \) is \( -\sqrt{2} \).
\end{example}

In general, \( N \)~can have principal orbits of the form \( \SU(2)/H \) with
\( H \) finite, and singular orbits~\( \SU(2)/K \) with \( K/H \) a sphere of
dimension three, one or zero.
If it is zero-dimensional, then the space is locally double covered by
\( \bR \times \SU(2)/H \), so we will not consider this case.
If \( K/H = S^{3} \), then \( K = \SU(2) = S^{3} \) and \( H \) is trivial.
In this case, a neighbourhood of the singular orbit is \( N = \bR^{4} \).
Using the techniques of Eschenburg \& Wang \cite{Eschenburg-W:initial}, one sees
that a two-form~\( \sigma_{i} \) as in~\eqref{eq:sigma} extends over such a
singular orbit at \( t = 0 \) only if each \( b_{pi}(t) = t^{2}f_{pi}(t) \) for
some smooth even functions~\( f_{pi}(t) \).
This two-form~\( \sigma_{i} \) is non-degenerate over the special orbit if and
only if there is a \( p \in \Set{1,2,3} \) so that \( f_{pi}(0) \ne 0 \).
Correspondingly a triple is weakly coherent in a neighbourhood of the special
orbit if and only if \( B = t^{2}F(t) \) with \( F^{T}F \) positive definite at
\( t = 0 \).

\begin{example}
  Consider
  \begin{equation*}
    B =
    \begin{pmatrix}
      t^{2} & t^{4} & 0 \\
      0 & t^{2} & t^{4} \\
      0 & 0 & t^{2}
    \end{pmatrix}
    = t^{2}F
    \quad\text{for}\ F =
    \begin{pmatrix}
      1 & t^{2} & 0 \\
      0 & 1 & t^{2} \\
      0 & 0 & 1
    \end{pmatrix}
    .
  \end{equation*}
  Then \( \left.(F^{T}F)\right|_{t=0} = \Id_{3} \) is positive definite.
  Additionally one may check that \( \partial_{t}(B^{T}B) > 0 \) for
  \( t \in (0,1) \).
  So we obtain a weakly coherent triple for \( t \in \halfopen{0,1} \).

  Taking \( C = (B^{T})^{-1}t^{4}S(t) \) for smooth symmetric \( S(t) \), we get
  curvature forms for a \( T^{3} \)-bundle and can construct corresponding
  \( \Spin(7) \)-geometries with \( T^{3} \times \SU(2) \) symmetry, which have
  stabiliser \( \SU(2) \) over the origin of~\( N \).

  The flexibility of the open condition~\eqref{eq:pd}, means that we can
  adjust~\( B \) for \( t > 0 \), so that the solution is defined on all of
  \( \halfopen{0,\infty} \).
  Note that the final \( \Spin(7) \)-structure is obtained by apply evolution
  equations to the total space of the \( T^{3} \)-bundle over~\( N \).
  The above does not guarantee existence for all time~\( s \) in this additional
  flow.
\end{example}

If \( K/H = S^{1} \), then a tubular neighbourhood of the singular orbit is
finitely covered by \( \SU(2) \times_{K_{0}} \bC \).
Writing \( X_{i} \) for the vector fields on \( \SU(2) \) dual to the
left-invariant one-forms~\( \theta_{i} \), we have that the action of
\( K_{0} \) is generated by \( Z = mX_{1} - n\partial_{\phi} \), where
\( z = te^{i\phi} \) is the coordinate on~\( \bC \).
Replacing \( Z \) by \( -Z \) we may assume that \( m > 0 \).
To get a non-principal orbit we must have \( n \ne 0 \); conjugating the
coordinate~\( z \) if necessary, we may assume that \( n > 0 \).
We then have \( L_{Z}\theta_{1} = 0 \) and \( L_{Z}\theta_{2} = -m\theta_{3} \).
For the two-form~\( \sigma_{i} \) of~\eqref{eq:sigma} to extend smoothly over
the singular orbit, \( b_{1i} \) must be an even function of~\( t \).
If \( b_{2i} \) and \( b_{3i} \) are identically zero, there are no further
restrictions.
If not, and this will be the case for at least one \( \sigma_{i} \) in a weakly
coherent triple, then they both must be of the form \( t^{m/n}f_{pi}(t) \),
\( p=2,3 \), with \( f_{pi}(t) \)~an even function of~\( t \).
The two-form \( \sigma_{i} \) is non-degenerate near \( t=0 \) if and only if
\( \lim_{t\to0} (\partial_{t}(b_{1i}^{2} + b_{2i}^{2} + b_{3i}^{2})/t) \) is
non-zero.
This is equivalent to either (i)~\( \ddot{b}_{1i}(0) \) and \( b_{1i}(0) \) are
both non-zero, or (ii)~one of \( \dot{b}_{2i}(0) \) or \( \dot{b}_{3i}(0) \) is
non-zero and \( m = n \).
Note that \( X_{1} \) has period~\( 4\pi \), so the neighbourhood of the
singular orbit is \( T^{*}\CP(1) \), equivariantly.
For a weakly coherent triple, we can always find a non-zero constant linear
combination of the \( \sigma_{i} \) for which the coefficient of
\( \theta_{2} \wedge \theta_{3} \) is zero at \( t = 0 \), case~(ii) occurs and
the local topology is restricted.
Furthermore, closure of \( \sigma_{i} \) implies
\( b_{1i}^{2} + b_{2i}^{2} + b_{3i}^{2} \) is monotone.
In particular, for weakly coherent triples there can be at most one singular
orbit.
For a weakly coherent triple to be non-degenerate over a singular orbit
\( G/K \) with \( \dim K \) requires \( \lim_{t\to0}(\partial_{t}(B^{T}B)/t) \)
to be positive definite.

\begin{example}
  We may take
  \begin{equation*}
    B =
    \begin{pmatrix}
      1 + \tfrac{1}{2}t^{2} & 0 & 0 \\
      t^{4} & t & 0 \\
      0 & t^{3} & t
    \end{pmatrix}
    .
  \end{equation*}
  This has
  \begin{equation*}
    \partial_{t}(B^{T}B) = t
    \begin{pmatrix}
      2 + t^{2} + 8t^{6} & 5t^{3} & 0 \\
      5t^{3} & 2 + 6t^{4} & 4t^{2} \\
      0 & 4t^{2} & 2
    \end{pmatrix}
    ,
  \end{equation*}
  so \( \partial_{t}(B^{T}B)/t \) converges to~\( 2\Id_{3} \), which is positive
  definite.
  Numerically one may check that \( \partial_{t}(B^{T}B) \) is positive definite
  on \( (0,0.778) \).
  Again one may choose \( C \) so we get curvature forms that extend over the
  singular orbit, and we may also ensure that one of these forms is non-zero in
  \( H^{2}(\CP(1)) \).
  As \( b_{2}(\CP(1)) = 1 \), we can rescale \( C \) by a constant so that the
  periods are integral.
  The resulting \( \Spin(7) \)-manifold will then have one-dimensional
  stabiliser over the singular orbit of~\( N \).
\end{example}

The above cases show that local examples of \( \Spin(7) \)-manifolds with
\( T^{3} \times \SU(2) \)-symmetry exist and exhibit all cases with rank-one
stabilisers.
When the stabiliser of point is of rank two, then this is is a critical point
for the multi-moment map~\( \nu_{4} \), so \( \nu_{4}^{-1}(s)/T^{3} \) is not
necessarily smooth and a more detailed analysis, or another approach, is
required.

The above examples may be specialised to give left-invariant coherent triples
\cite{Madsen-S:multi-moment}, so \( B \) is block diagonal with blocks of size
\( 1 \) and~\( 2 \), to produce \( G_{2} \)-manifolds with
\( T^{2} \times \SU(2) \)-symmetry and singular orbits with rank-one
stabilisers, showing that the construction of Aslan \& Trinca
\cite{Aslan-T:G2-cohom2} admits solutions.

\providecommand{\bysame}{\leavevmode\hbox to3em{\hrulefill}\thinspace}

\end{document}